\documentclass[12pt]{article}

\usepackage{amssymb,amsmath,amscd}
\usepackage{mathrsfs,mathtools}
\usepackage{amsfonts}
\usepackage{amssymb,amsthm}
\usepackage[arc,all]{xy}
\usepackage{tikz-cd}
\usepackage{hyperref}

\newcommand{\Z}{{\mathbb Z}}

\newcommand{\CC}{{\mathcal C}}

\newcommand{\EE}{{\mathcal E}}
\newcommand{\FF}{{\mathcal F}}

\newcommand{\II}{{\mathcal I}}

\newcommand{\Ascr}{{\mathscr A}}
\newcommand{\Bscr}{{\mathscr B}}
\newcommand{\Cscr}{{\mathscr C}}

\newtheorem{thm}{Theorem}[section]

\newtheorem{defn}[thm]{Definition}

\newcommand{\bcen}{\begin{center}}
\newcommand{\ecen}{\end{center}}

\def\dim{{\rm dim}}

\title{Calabi-Yau categories and\\graded quivers with potential}

\date{}
\author{Jie Ren}

\begin{document}
\maketitle

\begin{abstract}
We prove that the equivalence classes of d-dimensional Calabi-Yau $A_\infty$-categories (dCY category for short) of certain type are in one-to-one correspondence with the gauge equivalence classes of graded quivers with potential.
\end{abstract}

\section{Introduction}

The motivic Donaldson-Thomas theory of an ind-constructible 3-dimensional Calabi-Yau $A_\infty$-category is explored by Kontsevich-Soibelman in \cite{KoSo1}. In particular, if such a category is given by a quiver with potential, then the Donaldson-Thomas invariants (DT invariants for short) satisfy some remarkable properties. Some wall-crossing formulae are given by the mutations of the quiver with potential. In \cite{KoSo1} the authors proved that there is a one-to-one correspondence between the $A_\infty$-equivalence classes of certain type of \textbf{k}-linear 3CY categories and the gauge equivalence classes of quivers with potential.

In this paper we consider dCY categories, which are generated by a finite collection $\EE=\{E_{i}\}_{i\in I}$ of generators satisfying some properties. We prove that there is a one-to-one correspondence between the $A_\infty$-equivalence classes of this type of \textbf{k}-linear dCY categories and the gauge equivalence classes of graded quivers with potential. The generators correspond to the vertices of the quiver, and $Ext^k(E_i,E_j), k=1,\ldots,d-1$ give rise to the graded arrows. The deformation of a dCY category coincides with the gauge equivalence classes of minimal potentials of the quiver.
\\

%{\it Acknowledgements:} I would like to thank Jesse Huang for inspiring conversations.
%\section{Idea}

%The potential W, whose degree is 3-d, consists of the arrows in the "double quiver", i.e., quiver encoding information about $Ext^i,\;i=1,\ldots,d-1$, whose arrows are $a, deg(a)=0,\ldots,-\frac{d-2}{2}$ (d even, i.e., correspond to $Ext^k,\;k=1,\ldots,\frac{d}{2}$) or $-\frac{d-3}{2}$ (d odd, i.e., correspond to $Ext^k,\;k=1,\ldots,\frac{d-1}{2}$) and $a^*,deg(a^*)=-\frac{d-2}{2}\text{(d even, i.e., correspond to $Ext^k,\;k=\frac{d}{2},\ldots,d-1$) or} -\frac{d-1}{2}\text{(d odd, i.e., correspond to $Ext^k,\;k=\frac{d+1}{2},\ldots,d-1$)},\ldots,2-d$. To prove the compatibility of the construction of W with gauge group action (as in Section 8.1, Th 9), let the idea $\JJ$ generated by coordinates on $Ext^0,\; Ext^d$ (since no arrows corresponding to them) and $Ext^{d-1}$ (gives arrows $a^*$ of degree 2-d, thus con't contribute to W for degree reason). All the other coordinates can appear in W. For instance, an arrow of degree 3-d (corresponds to $Ext^{d-2}$) combine with arrows of degree 0 (correspond to $Ext^1$) can give terms in W.

%I guess the DGLA $\widehat{\mathfrak g}$ is shifted from the cyclic words by 2-d since we want the potential (degree 3-d) to be in $H^1(\mathfrak g_{can})$.

%According to van den Bergh, if W consists of only arrows of Q then $\{W,W\}=0$ is vacuous.

\section{d Calabi-Yau categories}

Let's recall the notion of d Calabi-Yau $A_\infty$-categories. For details please see \cite{F,K,K1,KoSo}.

\begin{defn}
An $A_\infty$-category $\CC$ over a field \textbf k consists of the data:
\begin{itemize}
\item[1)]The set of objects $\mathcal{M}=Ob(\CC)$.
\item[2)]The spaces of morphisms $Hom^\bullet$, which for any objects $E,F\in\mathcal M$ is a $\mathbb{Z}$-graded \textbf k-vector space $Hom^\bullet(E,F)=\oplus_{n\in\mathbb{Z}}Hom^n(E,F)$.
\item[3)]The higher composition maps $$m_n: Hom^{l_1}(E_1,E_2)\otimes\cdots\otimes Hom^{l_n}(E_n,E_{n+1})\rightarrow Hom^{l_1+\cdots+l_n+2-n}(E_1,E_{n+1})$$ for $n\geqslant1$ and $l_1,\ldots,l_n\in\mathbb{Z}$, which satisfy higher associativity property in the sense of $A_\infty$-categories. Equivalently, $(\sum_n m_n[1])^2=0$.
\end{itemize}
\end{defn}

Thus $m_1$ is a differential on $Hom^\bullet$, and we denote $$Ext^i(E,F):=H^i(Hom^\bullet(E,F),m_1).$$

\begin{defn}
An $A_\infty$-category has is weakly unital if $Ext^0(E,E)$ contains the identity morphism for any $E\in\mathcal M$.
\end{defn}

\begin{defn}
Let $\CC_1$, $\CC_2$ be two $A_\infty$-categories. An $A_\infty$ functor $\FF: \CC_1\rightarrow\CC_2$ is a collection $\FF_k,\;k\in\Z_{\geqslant0}$ such that
\begin{itemize}
\item[1)] $\FF_0: Ob(\CC_1)\rightarrow Ob(\CC_2)$ is a map of objects,
\item[2)] for $E,F\in Ob(\CC_1)$, $\FF_k(E,F):Hom_{\CC_1}^{l_1}(E,E_1)\otimes\cdots\otimes Hom_{\CC_1}^{l_n}(E_{n-1},F)\rightarrow Hom_{\CC_2}^{l_1+\cdots+l_n+1-n}(E,F)$,
\end{itemize}
and $\sum_k\FF_k[1]$ is a chain map with respect to $\sum_nm_n[1]$.

A functor $\FF:\CC_1\rightarrow\CC_2$ is called an equivalence if $\FF$ is a full embedding, i.e. it induces an isomorphism
$Ext^\bullet(E,E)\cong Ext^\bullet(\FF(E),\FF(F)), \forall E, F\in Ob(\CC_1)$ and moreover, there exists a map
$s:Ob(\CC_2)\rightarrow Ob(\CC_1)$ such that for any object $E\in Ob(\CC_2)$ we have $E\cong\FF(s(E))$.
\end{defn}

Assume that the field \textbf{k} has characteristic zero.

\begin{defn}
A Calabi-Yau category of dimension d is a weakly unital \textbf k-linear triangulated $A_\infty$-category $\mathcal{C}$, such that the $\mathbb{Z}$-graded vector space $Hom^\bullet(E,F)=\oplus_{n\in\mathbb{Z}}Hom^n(E,F)$ is finite-dimensional for any objects E and F. This implies that $Ext^\bullet(E,F)$ is also finite-dimensional. Moreover, we have the following data:
\begin{itemize}
\item[$\bullet$]A non-degenerate pairing
$$(\bullet,\bullet): Hom^\bullet(E,F)\otimes Hom^\bullet(F,E)\rightarrow{\rm k}[-d],$$
which is symmetric with respect to interchaging E and F.
\item[$\bullet$]A polylinear $\mathbb{Z}/N\mathbb{Z}$-invariant map $$W_N:\otimes_{1\leqslant i\leqslant N}(Hom^\bullet(E_i,E_{i+1})[1])\rightarrow{\rm k}[3-d],$$
for any $N\geqslant2$ and objects $E_1=E_{N+1},\ldots,E_N$.
\item[$\bullet$]The above maps are compatible in the sense of $$W_N(a_1,\ldots,a_N)=(m_{N-1}(a_1,\ldots,a_{N-1}),a_N).$$
\end{itemize}
\end{defn}

The collection $(W_N)_{N\leqslant2}$ is called the potential of $\mathcal{C}$.

\section{Main theorem}

Let $Q$ be a graded quiver with $I$ the set of vertices and $\Omega$ the set of arrows, i.e., there is a map $\Omega\rightarrow\Z$. The potential is an element $W\in\widehat{\textbf kQ}/\overline{[\textbf kQ,\textbf kQ]}$. The group of grading preserving continuous automorphisms of the completed path algebra $\widehat{\textbf kQ}$ of $Q$ preserving the vertices, acts on the set of potentials of $Q$, and is called the gauge action.

\begin{thm}
Let $\Cscr$ be a d-dimensional \textbf{k}-linear Calabi-Yau category generated by a finite collection $\EE=\{E_{i}\}_{i\in I}$ of generators satisfying
\begin{itemize}
\item[$\bullet$]$Ext^{0}(E_{i},E_{i})=\textbf{k} \cdot id_{E_{i}}$,
\item[$\bullet$]$Ext^{0}(E_{i},E_{j})=0, \forall i\neq j$,
\item[$\bullet$]$Ext^{<0}(E_{i},E_{j})=0, \forall i, j$.
\end{itemize}
The equivalence classes of such categories with respect to $A_{\infty}$-transformations preserving the Calabi-Yau structure and $\mathcal {E}$, are in one-to-one correspondence with the gauge equivalence classes of pairs $(\overline Q,W)$. Here $Q$ is a finite graded quiver (no cycles of length 2 whose arrows are both of degree $\frac{2-d}{2}$ if d is even) whose arrows are of degrees $\lfloor\frac{3-d}{2}\rfloor,\ldots,0$, and $\overline Q$ is the double quiver obtained by adding the dual arrows $a^*:j\rightarrow i$ of degree $2-d-r$ to $Q$ whenever there is an arrow $a:i\rightarrow j$ in $\Omega$ of degree $r$. And $W$ is a minimal potential (i.e. its Taylor decomposition starts with terms of degree at least 3) of degree $3-d$ satisfying $\{W,W\}=0$.
\end{thm}

\begin{proof}

Let's denoted by $\Ascr$ the set of equivalence classes of such $d$ Calabi-Yau categories, and
$\Bscr$ the set of equivalence classes of finite graded quivers with minimal potential.

Given such a category $\Cscr$, we associate a quiver $Q$ whose vertices $\{i\}_{i\in I}$ are in one-to-one correspondence with $\mathcal {E}=\{E_{i}\}_{i\in I}$. To give the arrows, note that $\dim Ext^k(E_i,E_j)=\dim Ext^{d-k}(E_j,E_i)^\vee=\dim Ext^{d-k}(E_j,E_i)$ since $\Cscr$ is $d$ Calabi-Yau. If $d$ is odd, then the number of arrows $a:i\rightarrow j$ of degrees $r=\frac{3-d}{2},\ldots,0$ is equal to $\dim Ext^{-r+1}(E_{i}, E_{j})$, and the dual arrows $a^*:j\rightarrow i$ are given by $Ext^k(E_j,E_i),k=\frac{d+1}{2},\ldots,d-1$, where $k=-|a^*|+1$. If $d$ is even, then the number of arrows $a:i\rightarrow j$ of degrees $r=\frac{4-d}{2},\ldots,0$ is equal to $\dim Ext^{-r+1}(E_{i}, E_{j})$, and the dual arrows $a^*:j\rightarrow i$ are given by $Ext^k(E_j,E_i),k=\frac{d+2}{2},\ldots,d-1$, where $k=-|a^*|+1$. In degree $\frac{2-d}{2}$, for any pair of distinct vertices $i$ and $j$ one can chose the orientation of the arrows such that the number of arrows $a:i\rightarrow j$ is $\dim Ext^{\frac{d}{2}}(E_i,E_j)$ and the number of arrows $a^*:j\rightarrow i$ is $\dim Ext^{\frac{d}{2}}(E_j,E_i)$, or the other way around (the orientation in this degree doesn't matter since $\dim Ext^{\frac{d}{2}}(E_i,E_j)=\dim Ext^{\frac{d}{2}}(E_j,E_i)$); the number of loops in $Q$ at any $i\in I$ of degree $\frac{2-d}{2}$ is equal to $\frac{1}{2}\dim Ext^{\frac{d}{2}}(E_i,E_i)$ (one can do this because the supersymmetric non-degenerate pairing on $Ext^{\bullet}(E_{i},E_{i})$ leads to a symplectic pairing on $Ext^{\frac{d}{2}}(E_{i},E_{i})$, thus $\dim Ext^{\frac{d}{2}}(E_{i},E_{i})$ is even). 
%\{The quiver obtained in this way when $d$ is even has no cycles of length 2 whose arrows are both of degree $\frac{2-d}{2}$(why do we need this?)\}. 
The restriction of the potential of $\Cscr$ to $Ext^k,k=1,\ldots,d-1$ gives $W$. This construction defines a map $\Phi: \Ascr\rightarrow\Bscr$.

To prove that $\Phi$ is a bijection, we consider $\mathscr{C}$ with a single generator $E$, and quiver $Q$ with a single vertex for simplicity. The general case can be proved in a similar way.

Let $Q$ be a quiver with one vertex and loops of degrees $r=\lfloor\frac{3-d}{2}\rfloor,\ldots,0$, and $\overline Q$ its double quiver endowed with a minimal potential $W_0$. We will construct a $d$ Calabi-Yau category with one generator $E$, such that the numbers of loops of various degrees coincide with the above construction. We will find an explicit formula for the potential on $A=Hom^\bullet(E,E)$. Let's consider the graded vector space
\begin{equation}
\begin{array}{ll}
Ext^\bullet(E,E)[1]&=Ext^0(E,E)[1]\oplus Ext^1(E,E)\oplus\cdots\oplus Ext^d(E,E)[-d+1]\\&\cong\textbf{k}[1]\oplus Ext^1(E,E)\oplus\cdots\oplus\textbf{k}[-d+1].
\end{array}
\end{equation}

We introduce graded coordinates on $Ext^{\bullet}(E,E)[1]$:
\begin{itemize}
\item[a)]the coordinate $\alpha$ of degree 1 on $Ext^0(E,E)[1]$,
\item[b)]the coordinate $\beta$ of degree $-d+1$ on $Ext^d(E,E)[-d+1]$,
\item[c)]
 \begin{itemize}
 \item[$\bullet$]when $d$ is odd, the coordinates $\{x_{r,n},\;n=1,\ldots,\dim Ext^{-r+1}(E,E)\}$ of degree $r$ on $Ext^{-r+1}(E,E)[r]$ for $r=\frac{3-d}{2},\ldots,0$,
 \item[$\bullet$]when $d$ is even, the coordinates $\{x_{r,n},\;n=1,\ldots,\dim Ext^{-r+1}(E,E)\}$ of degree $r$ on $Ext^{-r+1}(E,E)[r]$ for $r=\frac{4-d}{2},\ldots,0$, and the coordinates $\{x_{\frac{2-d}{2},n},\;n=1,\ldots,\frac{1}{2}\dim Ext^{\frac{d}{2}}(E,E)\}$ of degree $\frac{2-d}{2}$ on $Ext^{\frac{d}{2}}(E,E)[\frac{2-d}{2}]$,
 \end{itemize}
\item[d)]
 \begin{itemize}
 \item[$\bullet$]when $d$ is odd, the coordinates $\{\xi_{r,n},\;n=1,\ldots,\dim Ext^{-r+1}(E,E)\}$ of degree $r$ on $Ext^{-r+1}(E,E)[r]$ for $r=2-d,\ldots,\frac{1-d}{2}$,
 \item[$\bullet$]when $d$ is even, the coordinates $\{\xi_{r,n},\;n=1,\ldots,\dim Ext^{-r+1}(E,E)\}$ of degree $r$ on $Ext^{-r+1}(E,E)[r]$ for $r=2-d,\ldots,\frac{-d}{2}$, and the coordinates $\{\xi_{\frac{2-d}{2},n},\;n=1,\ldots,\frac{1}{2}\dim Ext^{\frac{d}{2}}(E,E)\}$ of degree $\frac{2-d}{2}$ on $Ext^{\frac{d}{2}}(E,E)[\frac{2-d}{2}]$,
 \end{itemize}
\end{itemize}

The Calabi-Yau structure on $A$ gives rise to the minimal potential $W=W(\alpha, x_{r,n}, \xi_{r,n}, \beta)$, which is a series in cyclic words on the space $Ext^{\bullet}(E,E)[1]$. If it arises from the pair $(\overline Q,W_0)$, then the restriction $W(0,x_{r,n},\xi_{r\neq2-d,n},0)$ must coincide with $W_0=W_0(x_{r,n},\xi_{r\neq2-d,n})$. Furthermore, $A$ defines a non-commutative formal  pointed graded manifold endowed with a symplectic structure (c.f. \cite{KoSo}). The potential $W$ satisfies the equation $\{W,W\}=0$, where $\{\bullet,\bullet\}$ is the corresponding Poisson bracket. So we need to construct an extension of $W_0$ to the formal series $W$ of degree $3-d$ in cyclic words on the graded vector space $Ext^{\bullet}(E,E)[1]$, satisfying $\{W,W\}=0$ with respect to the Poisson bracket
$$\{f,g\}=\sum_{r,n}\left[\frac{\partial}{\partial x_{r,n}},\frac{\partial}{\partial\xi_{2-d-r,n}}\right](f,g)+\left[\frac{\partial}{\partial\alpha},\frac{\partial}{\partial\beta}\right](f,g).$$

Let's start with $W_{can}=\alpha^{2}\beta+\sum_{r,n}(\alpha x_{r,n}\xi_{2-d-r,n}-\alpha\xi_{2-d-r,n}x_{r,n})$. This potential makes
$Ext^{\bullet}(E,E)$ into a $d$ Calabi-Yau algebra with associative product and the unit. Any minimal potential $W_0$ on $\bigoplus_{k=1}^{d-2}Ext^k(E,E)$ can be lifted to the minimal potential on $Ext^\bullet(E,E)$ by setting
$W:=W_{can}+W_0$. We claim that $\{W,W\}=0$. Indeed,
\begin{equation}
\begin{array}{ll}
\{W_{can},W_{can}\}&=\sum\limits_{r,n}[\frac{\partial
W_{can}}{\partial x_{r,n}},\frac{\partial W_{can}}{\partial\xi_{2-d-r,n}}]+[\frac{\partial W_{can}}{\partial\alpha},\frac{\partial W_{can}}{\partial\beta}]\\&=\sum\limits_{r,n}[(\xi_{2-d-r,n}\alpha-\alpha\xi_{2-d-r,n})(\alpha
x_{r,n}-x_{r,n}\alpha)\\&-(\alpha
x_{r,n}-x_{r,n}\alpha)(\xi_{2-d-r,n}\alpha-\alpha\xi_{2-d-r,n})]\\&+(\alpha\beta+\beta\alpha+\sum\limits_{r,n}(x_{r,n}\xi_{2-d-r,n}-\xi_{2-d-r,n}x_{r,n}))\alpha^2\\&-\alpha^2(\alpha\beta+\beta\alpha+\sum\limits_{r,n}(x_{r,n}\xi_{2-d-r,n}-\xi_{2-d-r,n}x_{r,n}))
\\&=0,\nonumber
\end{array}
\end{equation}
and $\{W_0,W_0\}=0$ holds trivially. Moreover, $$\{W_{can},W_0\}=\alpha\sum_{r\neq2-d,n}\left[\frac{\partial W_0}{\partial\xi_{2-d-r,n}},\xi_{2-d-r,n}\right]-\alpha\sum_{r\neq2-d,n}\left[x_{r,n},\frac{\partial W_0}{\partial x_{r,n}}\right]=0$$ since $\sum_{r\neq2-d,n}[\frac{\partial W_0}{\partial\xi_{2-d-r,n}},\xi_{2-d-r,n}]=\sum_{r\neq2-d,n}[x_{r,n},\frac{\partial W_0}{\partial x_{r,n}}]=0$.

Next we need to check compatibility of the above construction with the gauge group action for $d>2$ (there's no $W_0$ when $d=2$). Let $G_0$ be the subgroup of the grading preserving automorphisms of the group of continuous automorphisms of the algebra of formal series $\textbf{k}\langle\langle\alpha,x_{r,n},\xi_{r,n},\beta\rangle\rangle$, and $\II\subset\textbf{k}\langle\langle\alpha,x_{r,n},\xi_{r,n},\beta\rangle\rangle$ a closed two-sided ideal generated by $\alpha$, $\beta$ and $\xi_{2-d,n}$. Since the group $G_0$ preserves $\II$, we obtain a homomorphism of groups $G_0\rightarrow Aut(\textbf{k}\langle\langle x_{r,n},\xi_{r\neq2-d,n}\rangle\rangle)$.

The above construction from $Q$ to $\mathscr{C}$ shows that $\Phi$ is a surjection.

Finally, we need to check that $\Phi$ is an injection. The $d$
Calabi-Yau algebras we are considering can be thought of as
deformations of the $d$ Calabi-Yau algebra
$A_{can}=Ext^{\bullet}(E,E)$ corresponding to the potential
$W_{can}$. The deformation theory of $A_{can}$ is controlled by a differential graded Lie algebra (DGLA) $\mathfrak{g}_{can}=\bigoplus_{n\in\mathbb{Z}}\mathfrak{g}_{can}^n$, which is a DG Lie subalgebra of the DGLA
$\widehat{\mathfrak{g}}=(\prod_{k\geqslant1}Cycl^k(A_{can}[1])^{\vee})[2-d]=\bigoplus_{n\in\mathbb{Z}}\widehat{\mathfrak{g}}^n$. Here we write $\widehat{\mathfrak{g}}^n=\{w|coh.deg\:w=n\}$, and $\mathfrak{g}_{can}^n=\{w\in\widehat{\mathfrak{g}}^{n}|cyc.deg\:w\geqslant n+2\}$, where coh.deg means the cohomological degree of $W$, and cyc.deg means the number of letters $\alpha,\:x_{r,n},\:\xi_{r,n},\:\beta$ that $W$ contains. In these DGLAs, the Lie bracket is given by the Poisson bracket and the differential is given by $D=\{W_{can}, \bullet\}$. The set of $A_\infty$-equivalence classes of minimal $d$ Calabi-Yau agebras can be identified with the set of gauge equivalence classes of solutions $\gamma\in \mathfrak{g}^1_{can}$ to the Maurer-Cartan equation $D\gamma+\frac{1}{2}[\gamma,\gamma]=0$. The DGLA $\mathfrak{g}_{can}$ is indeed a DG Lie subalgebra of $\widehat{\mathfrak{g}}$ since $D$ increases both coh.deg and cyc.deg by 1 and the Lie bracket preserves the difference between coh.deg and cyc.deg as well. As vector spaces,
$\widehat{\mathfrak{g}}=\mathfrak{g}_{can}\bigoplus\mathfrak{g}$,
where $\mathfrak{g}=\bigoplus_{n\in\mathbb{Z}}\mathfrak{g}^{n}$, and $\mathfrak{g}^{n}=\{w\in\widehat{\mathfrak{g}}^{n}|cyc.deg\:w<n+2\}$. For the same reason as $\mathfrak{g}_{can}$, we have that
$\mathfrak{g}$ is also a DG Lie subalgebra of $\widehat{\mathfrak{g}}$. It follows that $\mathfrak{g}_{can}$ is a
direct summand of the complex $\widehat{\mathfrak{g}}$. The latter is quasi isomorphic to the dual cyclic complex $CC_{\bullet}(A_{can})^{\vee}$. Let $A_{can}^{+}\subset A_{can}$ be the non-unital $A_{\infty}$-subalgebra consisting of terms of positive cohomological degree. Then for cyclic homology, $HC_{\bullet}(A_{can})\simeq HC_{\bullet}(A_{can}^{+})\bigoplus HC_{\bullet}(\textbf{k})$. In terms of dual complex $\widehat{\mathfrak{g}}$, this isomorphism means the decomposition into a direct sum of the space of cyclic series in variables $x_{r,n},\:\xi_{r,n},\:\beta$ (corresponds to $HC_{\bullet}(A_{can}^{+})^{\vee}$), and the one in variable $\alpha$ (corresponds to $HC_{\bullet}(\textbf{k})^{\vee}$). We see that the series in $\alpha$ don't contribute to the cohomology of $\mathfrak{g}_{can}$ because of the degree reason.
%[($\alpha^n$ has cohomologycal degree n+d-2 but the cyclic degree is not bigger than or equal to n+d-2+2), (since $\{W_{can}, \alpha\}=-\alpha^{2}$)].
Moreover, the cohomological degree of series in $x_{r,n},\:\xi_{r,n},\:\beta$ is non-positive. Recall that we shifted the cohomological grading in the Lie algebras by $2-d$, hence $H^{>d-2}(\mathfrak{g}_{can})=0$.

The set of gauge equivalence classes of minimal potentials of $\overline Q$ can be identified with the set of gauge equivalence classes of solutions to the Maurer-Cartan equation in the DGLA $\mathfrak h=\mathfrak h^0\oplus\mathfrak h^1$, where $\mathfrak h^0=\{w\in\widehat{\mathfrak g}^0|w\text{ consists of }x_{r,n}, \xi_{r,n}\}$, and $\mathfrak h^1=\{w\in\widehat{\mathfrak g}^1|w\text{ consists of }x_{r,n},\xi_{r\neq2-d,n}\}$. The differential is trivial and the Lie bracket is the above Poisson bracket. Here $\mathfrak h^0$ is identified with the Lie algebra of continuous derivations of the topological algebra $\textbf{k}\langle\langle x_{r,n},\xi_{r,n}\rangle\rangle$ preserving the augmentation ideal $(x_{r,n},\xi_{r,n})$, and we identify $\mathfrak h^1$ with the $\mathfrak h^0$-module of minimal cyclic potentials on $A^1$. The above construction of $W=W_{can}+W_0$ gives rise to a homomorphism of DGLAs $\Psi: \mathfrak h\rightarrow\mathfrak g_{can}$, which induces maps on cohomology $H^i(\Psi):H^i(\mathfrak h)\rightarrow H^i(\mathfrak g_{can})$. The map $H^1(\Psi)$ is an isomorphism, the map $H^2(\Psi)$ is injective and $H^0(\Psi)$ is surjective, thus the corresponding deformation theories are isomorphic by \cite[Th.~3.1]{M}. This proves the injectivity of $\Phi$.

%(If there's no $\xi$'s in W then we argue as follows(but we probably do need $\xi$'s):)The set of gauge equivalence classes of minimal potentials of $Q$ can be identified with the set of gauge equivalence classes of solutions to the Maurer-Cartan equation in the DGLA $\mathfrak h=\mathfrak h^0\oplus\mathfrak h^1$, where $\mathfrak h^0=\{w\in\widehat g^0|w\:consists\:of\:exactly\:of\:\xi_{r,n}\:and\:at\:least\:one\:of\:x_{r,n}\}$, and $\mathfrak h^1=\{w\in\widehat{\mathfrak g}^1|w\:consists\:of\:x_{r,n}\}$. Here $\mathfrak h^0$ is identified with the Lie algebra of continuous derivations of the topological algebra $\textbf{k}\langle\langle x_{r,n}\rangle\rangle$ preserving the augmentation ideal $(x_{r,n})$, and we identify $\mathfrak h^1$ with the $\mathfrak h^0$-module of minimal cyclic potentials on $A^1$. The above construction of $W=W_{can}+W_0$ gives rise to a homomorphism of DGLAs $\Psi: \mathfrak h\rightarrow\mathfrak g_{can}$.

\end{proof}

\end{document}